\documentclass{article}

\usepackage{amsmath, amssymb, amsthm}
\usepackage{enumitem}
\usepackage{comment}
\usepackage[margin=1.15in]{geometry}
\usepackage[numbers,sort&compress]{natbib}

\newtheorem{theorem}{Theorem}[section]
\newtheorem{lemma}[theorem]{Lemma}
\newtheorem{proposition}[theorem]{Proposition}
\newtheorem{corollary}[theorem]{Corollary}
\newtheorem*{claim}{Claim}
\theoremstyle{definition}
\newtheorem{definition}[theorem]{Definition}
\newtheorem{example}[theorem]{Example}
\theoremstyle{remark}
\newtheorem{remark}[theorem]{Remark}

\theoremstyle{definition}
\newtheorem{innercustomthm}{Theorem}
\newenvironment{customthm}[1]
  {\renewcommand\theinnercustomthm{\textbf{#1}}\innercustomthm}
  {\endinnercustomthm}

\newcommand{\Z}{\mathbb{Z}}
\newcommand{\Q}{\mathbb{Q}}
\newcommand{\BS}{\mathrm{BS}}
\newcommand{\TF}{\text{-}\mathrm{TF}}
\newcommand{\TCF}{\text{-}\mathrm{TCF}}
\newcommand{\ACFA}{\mathrm{ACFA}}
\newcommand{\Stab}{\operatorname{Stab}}
\newcommand{\Frac}{\operatorname{Frac}}
\title{Baumslag--Solitar Subgroups Obstruct Model Companions for Fields with Group Actions}
\author{Makoto Yanagawa}
\date{}
\begin{document}
\maketitle

\begin{abstract}
  We prove that if a group \(G\) contains a Baumslag--Solitar group \(\mathrm{BS}(m,n)\), with \(mn\neq0\), then the theory of fields with a \(G\)-action has no model companion.
  In particular, every group containing \(\mathbb{Z}^2\cong\mathrm{BS}(1,1)\) has no model companion for its field actions, resolving a conjecture of Beyarslan and Kowalski.
  Consequently, there are no model companions for field actions of \(\mathbb{Q}^2\), Thompson's groups \(F,T,V\), or Higman's group.
  Our argument adapts Hrushovski's method for two commuting automorphisms.
  The key modification is the cyclotomic condition \(\theta_q\), which removes the need to prescribe a non-trivial action on a root of unity and thereby makes Hrushovski's method applicable after passing from a subgroup to an ambient group.
\end{abstract}

\section{Introduction}

A single automorphism of a field is one of the classical success stories of model theory.
Macintyre proved that the theory of fields with one distinguished automorphism has a model companion \cite{Macintyre1997}; the model theory of this companion was developed by Chatzidakis and Hrushovski \cite{ChatzidakisHrushovski1999}, and the resulting theory is now denoted by \(\ACFA\).
The picture changes abruptly when one adds a second commuting automorphism.
Hrushovski observed that fields with two commuting automorphisms have no model companion, and a proof was recorded by Kikyo \cite{Kikyo2004}.
Thus already
\[
  \Z\TCF \text{ exists},\quad \Z^2\TCF \text{ does not exist}.
\]

This contrast led to the systematic study of fields with an action of a fixed group.
Sj\"ogren treated actions of finite groups and free groups in his 2005 report \cite{Sjogren2005}; Medvedev established the existence of the model companion for \(\Q\)-actions \cite{Medvedev2015}.
More recently, Beyarslan and Kowalski classified the torsion abelian case \cite{BKTorsion2023} and proved that model companions are rare among finitely generated acting groups \cite{BKRare2024}.
The latter result also led to the withdrawal of an earlier claim concerning virtually free groups; see the corrigendum \cite{BKCorrigendum2024}.
Their Frattini-kernel criterion gives non-existence for many individual finitely generated acting groups, including \(\Z^2\) itself \cite[Theorem~3.7 and Corollary~4.3]{BKRare2024}.
It does not, however, provide a general mechanism by which an obstruction for a subgroup persists in an arbitrary ambient group.

Several results suggested that such obstructions might pass to ambient groups.
Beyarslan and Kowalski proved non-existence for \(\Z^2*H\), with \(H\) finitely generated \cite[Corollary~4.5]{BKRare2024}, and for \(G\times C_2\) whenever \(\Z^2\leq G\) \cite[Corollary~6.9]{BKTorsion2023}, but left open the conjecture
\[
  \Z^2\leq G \quad\Longrightarrow\quad G\TCF \text{ does not exist}.
\]
This is \cite[Conjecture~6.6]{BKTorsion2023}.
More generally, they asked whether \(H\leq G\) and the existence of \(G\TCF\) imply the existence of \(H\TCF\) \cite[Question~4.8]{BKRare2024}.

There is a concrete reason why the statement for groups containing \(\Z^2\) remained a conjecture in \cite{BKTorsion2023}.
The available obstruction was Hrushovski's argument for two commuting automorphisms.
In the strengthened form recalled in \cite[Theorem~6.7]{BKTorsion2023}, it excludes an \(\aleph_0\)-saturated p.e.c. \(\Z^2\)-field (Definition~\ref{def:pec}) only under a prescribed cyclotomic condition: a primitive cube root of unity \(\zeta\) must be present and both generators must act on it non-trivially by
\[
  \zeta\longmapsto\zeta^2.
\]
Although p.e.c. descends from \(G\) to a subgroup, the induced \(\Z^2\)-action on an arbitrary \(G\)-field need not satisfy this additional condition.
This is precisely the obstruction identified by Beyarslan and Kowalski in their discussion following \cite[Theorem~6.7]{BKTorsion2023}; in particular, they note that for \(G=\Q^2\) it can never be satisfied.

We modify Hrushovski's method by imposing the cyclotomic condition \(\theta_q\) of Definition~\ref{def:cyclotomic-condition}.
This condition replaces the prescribed non-trivial action on a root of unity by a compatibility condition for the Baumslag--Solitar relation and permits the root of unity to be fixed.
After replacing the stable letter \(\tau\) by \(\tau^{q-1}\), Fermat's little theorem makes the relevant cyclotomic exponent equal to \(1\).
A root of unity fixed by the ambient group can therefore be used, so the construction applies after passing from a subgroup to an ambient group.
The corresponding realization problem is given in Definition~\ref{def:realizationProblem}.

Introduced by Baumslag and Solitar in connection with non-Hopfian one-relator groups \cite{BaumslagSolitar1962}, the Baumslag--Solitar groups are
\[
  \BS(m,n) := \langle \sigma,\tau\mid\tau\sigma^m\tau^{-1}=\sigma^n\rangle,\quad mn\neq0.
\]
The family contains \(\BS(1,1)\cong\Z^2\), the Klein bottle group \(\BS(1,-1)\), the solvable groups \(\BS(1,n)\cong\Z[1/n]\rtimes_n\Z\), and non-solvable examples such as \(\BS(2,3)\).

The main theorem passes the obstruction from a Baumslag--Solitar subgroup to the ambient group.

\begin{customthm}{\ref{th:overgroupBSmn}}
  If \(\BS(m,n)\leq G\) for some non-zero integers \(m,n\), then \(G\TCF\) does not exist.
\end{customthm}

Taking \(m=n=1\) settles \cite[Conjecture~6.6]{BKTorsion2023}: every group containing \(\Z^2\) has a theory of field actions with no model companion.
In particular, \(\Q^2\TCF\) does not exist, resolving the example for which the earlier cyclotomic obstruction can never be satisfied.
Further applications include Thompson's groups \(F,T,V\) and Higman's group; see Example~\ref{ex:furtherExamples}.
The theorem removes both the finite-generation hypothesis on the second factor in the free-product result and the auxiliary factor \(C_2\) in the direct-product result.
More generally, it answers the contrapositive form of \cite[Question~4.8]{BKRare2024} for Baumslag--Solitar subgroups, without any residual finiteness, finite generation, or presentation hypothesis on the ambient group.

These examples also clarify the difference from the Frattini-kernel criterion of \cite[Theorem~3.7]{BKRare2024}.
For Thompson's group \(F\), every finite quotient factors through the abelianization \(F_{\mathrm{ab}}\cong\Z^2\), since every proper quotient of \(F\) is abelian \cite[Theorems~4.1 and 4.3]{CannonFloydParry1996}; hence \(\widehat F\cong\widehat{\Z^2}\), and that criterion already applies.
By contrast, Thompson's groups \(T,V\) are infinite simple groups, while Higman's group has no non-trivial finite quotient \cite{CannonFloydParry1996,RivasTriestino2019}.
Their profinite completions, and therefore the kernels of their universal Frattini covers, are trivial, so the Frattini-kernel criterion gives no obstruction for these groups, whereas Theorem~\ref{th:overgroupBSmn} applies through their Baumslag--Solitar subgroups.

The proof has two steps.
A finite periodic \(\sigma\)-orbit produces the incompatibility recorded in Proposition~\ref{prop:realizationObstruction}, while Proposition~\ref{prop:SigmaRealization} realizes every finite fragment in a regular extension.
Compactness, \(\aleph_0\)-saturation, and p.e.c. then give the contradiction.

The paper is organized as follows.
Section~\ref{sec:preliminaries} records the preliminaries, including descent to subgroups and the quotient-field correspondence.
Sections~\ref{sec:BS-cocycles}--\ref{sec:realization-construction} develop the cocycles and realization construction.
Section~\ref{sec:BS-obstruction} proves the Baumslag--Solitar obstruction for suitable homomorphic images, and Section~\ref{sec:overgroups} applies it to groups containing a Baumslag--Solitar group.

\section{Preliminaries}\label{sec:preliminaries}

\subsection{Fields with group actions and model-theoretic terminology}

Let \(G\) be a group.
Write
\[
  L_G:=L_{\mathrm{ring}}\cup\{\sigma_g:g\in G\},
\]
where each \(\sigma_g\) is a unary function symbol.
A \emph{\(G\)-field} is an \(L_G\)-structure whose reduct to \(L_{\mathrm{ring}}\) is a field and in which
\[
  \sigma_1=\operatorname{id},\quad \sigma_{gh}=\sigma_g\circ\sigma_h,
\]
and every \(\sigma_g\) is a field automorphism.
A \(G\)-field extension \(L/K\) is a field extension for which the \(G\)-action on \(L\) restricts to the given action on \(K\).
We denote by \(G\TF\) the first-order theory of \(G\)-fields.

We recall the model-theoretic terminology used below.
A theory \(T^*\) is a \emph{model companion} of a theory \(T\) if \(T^*\) is model complete and \(T\) and \(T^*\) are companions, equivalently, they have the same universal consequences and every model of either theory embeds into a model of the other.
When the model companion of \(G\TF\) exists, we denote it by \(G\TCF\).
A model \(M\models T\) is \emph{existentially closed} (e.c.) in the class of models of \(T\) if, whenever \(M\subseteq N\models T\), every existential formula with parameters from \(M\) which holds in \(N\) already holds in \(M\).
In particular, when \(G\TCF\) exists, its models are existentially closed \(G\)-fields.
See, for example, \cite[\S\S3.1 and 3.4]{Marker2002} for model completeness and model companions.

For an infinite cardinal \(\kappa\), a structure \(M\) is \emph{\(\kappa\)-saturated} if every complete type in finitely many variables over a parameter set \(A\subseteq M\) with \(|A|<\kappa\) is realized in \(M\).
We shall only need \(\aleph_0\)-saturation.
We also use the Compactness Theorem in its standard form: a set of first-order sentences is satisfiable provided every finite subset is satisfiable.
Thus a partial type over a finite parameter set which is finitely satisfiable is consistent, and an \(\aleph_0\)-saturated structure realizes it after extension to a complete type.
For compactness, types, and saturation, see \cite[\S2.1, \S4.1, \S4.3]{Marker2002}.

If a group \(G\) acts on an abelian group \((A,\star)\), a map \(f:G\to A\) is a \emph{\(1\)-cocycle} if
\[
  f(gh)=f(g)\star g\bigl(f(h)\bigr).
\]
For the trivial action, this is precisely a homomorphism.

\subsection{Regular extensions and pseudo-existential closedness}

A field extension \(L/K\) is \emph{regular} if it is separable and \(K\) is relatively algebraically closed in \(L\); equivalently, \(L\otimes_K\overline K\) is an integral domain for an algebraic closure \(\overline K\) of \(K\).
We use freely the standard equivalences between regularity, linear disjointness from an algebraic closure, and geometric integrality; see \cite[Chapter~3]{FriedJarden2023}.
The \emph{characteristic exponent} of a field is \(1\) in characteristic zero and \(p\) in characteristic \(p>0\).

\begin{definition}\label{def:pec}
  A \(G\)-field \(K\) is \emph{pseudo-existentially closed} (p.e.c.) if \(K\) is existentially closed in every regular \(G\)-field extension of \(K\).
\end{definition}

Every existentially closed \(G\)-field is p.e.c.
The point of p.e.c. in the present argument is that regular extensions can be induced from subgroups.
\subsection{Descent to subgroups}

The following standard induction construction extends a regular \(H\)-field extension to a regular \(G\)-field extension.
Essentially the same tensor-product construction is used by Beyarslan and Kowalski in the proof of \cite[Proposition~2.2]{BKTorsion2023} to show that p.e.c. descends from a group to its subgroups.
We include the argument for completeness, using a section \(G/H\to G\) and the map \(\alpha\) defined below.

\begin{lemma}\label{lem:inducedRegularExtension}
  Let \(G\) be a group, let \(K\) be a \(G\)-field, let \(H\leq G\), and let \(L/K\) be a regular \(H\)-field extension.
  Then there exist a regular \(G\)-field extension \(M/K\) and a \(K\)-embedding of \(H\)-fields
  \[
    L\hookrightarrow M.
  \]
\end{lemma}

\begin{proof}
  Choose a section
  \[
    \rho:G/H\longrightarrow G
  \]
  of the natural projection such that \(\rho(H)=1\), and define
  \[
    \alpha:G\times G/H\longrightarrow H, \quad \alpha(\gamma,\lambda) := \rho(\gamma\lambda)^{-1}\gamma\rho(\lambda).
  \]
  Then
  \[
    \gamma\rho(\lambda) = \rho(\gamma\lambda)\alpha(\gamma,\lambda),
  \]
  and
  \[
    \alpha(\gamma\delta,\lambda) = \alpha(\gamma,\delta\lambda)\alpha(\delta,\lambda)
  \]
  for all \(\gamma,\delta\in G\) and \(\lambda\in G/H\).

  For each \(\lambda\in G/H\), let \(L_\lambda\) be a copy of \(L\), with elements denoted by \([z]_\lambda\) for \(z\in L\), and give \(L_\lambda\) the \(K\)-algebra structure
  \[
    k\longmapsto [\rho(\lambda)^{-1}(k)]_\lambda.
  \]
  Put
  \[
    R:=\bigotimes_{\lambda\in G/H}L_\lambda,
  \]
  where the tensor product is taken over \(K\) and is understood as the directed union of the tensor products over finite subsets of \(G/H\).

  For \(\gamma\in G\), define
  \[
    \gamma:L_\lambda\longrightarrow L_{\gamma\lambda}, \quad \gamma([z]_\lambda) := [\alpha(\gamma,\lambda)(z)]_{\gamma\lambda}.
  \]
  Since
  \[
    \gamma\rho(\lambda) = \rho(\gamma\lambda)\alpha(\gamma,\lambda),
  \]
  these maps are compatible with the action of \(\gamma\) on \(K\), and hence induce a ring automorphism of \(R\).
  The identity above shows that these automorphisms define a \(G\)-action on \(R\).

  Each \(L_\lambda/K\) is isomorphic, as a field extension of \(K\), to a twist of the regular extension \(L/K\), and is therefore regular.
  We claim that every finite tensor product of the \(L_\lambda\) over \(K\) is an integral domain.
  Indeed, suppose that \(T=L_{\lambda_1}\otimes_K\cdots\otimes_K L_{\lambda_{r-1}}\) is a domain.
  The injection \(T\hookrightarrow\Frac(T)\) remains injective after tensoring over \(K\) with the field \(L_{\lambda_r}\).
  Since \(L_{\lambda_r}/K\) is regular,
  \[
    L_{\lambda_r}\otimes_K\Frac(T)
  \]
  is a domain by the tensor-product criterion for regularity.
  Thus \(T\otimes_K L_{\lambda_r}\) is a domain, proving the claim by induction.
  Hence \(R\) is an integral domain.
  Put
  \[
    M:=\Frac(R).
  \]
  The \(G\)-action on \(R\) extends uniquely to \(M\).

  More generally, the same induction works after tensoring with any field extension \(\Omega/K\): at the last step one embeds into
  \[
    L_{\lambda_r}\otimes_K\Frac\bigl(T\otimes_K\Omega\bigr),
  \]
  which is a domain because \(L_{\lambda_r}/K\) is regular.
  Thus every finite tensor product above is geometrically integral over \(K\).
  Taking \(\Omega=\overline K\) shows that \(R\otimes_K\overline K\) is an integral domain.
  Since \(M=\Frac(R)\), one has
  \[
    M\otimes_K\overline K\cong (R\otimes_K\overline K)_{R\setminus\{0\}}.
  \]
  The image of every element of \(R\setminus\{0\}\) is nonzero, so this localization is an integral domain.
  Hence \(M/K\) is regular.
  This is precisely where regularity, rather than mere integrality of the factors, is used; compare the standard tensor criterion recalled above and \cite[Chapter~3]{FriedJarden2023}.

  Finally, since \(\rho(H)=1\), the natural embedding
  \[
    L=L_H\longrightarrow R\longrightarrow M
  \]
  is a \(K\)-embedding.
  Moreover, for \(h\in H\),
  \[
    \alpha(h,H)=h,
  \]
  and therefore
  \[
    h([z]_H)=[h(z)]_H.
  \]
  Thus \(L\hookrightarrow M\) is an embedding of \(H\)-fields.
\end{proof}

\begin{corollary}\label{cor:Descent}
  Let \(K\) be a pseudo-existentially closed \(G\)-field and let \(H\leq G\).
  Then \(K\), with the action restricted to \(H\), is a pseudo-existentially closed \(H\)-field.
  In particular, the restriction to \(H\) of an existentially closed \(G\)-field is a pseudo-existentially closed \(H\)-field.
\end{corollary}

\begin{proof}
  Let \(L/K\) be a regular \(H\)-field extension, and let \(\varphi(x,a)\) be a quantifier-free \(L_H\)-formula with parameters \(a\) from \(K\) such that
  \[
    L\models\exists x\,\varphi(x,a).
  \]
  By Lemma~\ref{lem:inducedRegularExtension}, there exist a regular \(G\)-field extension \(M/K\) and a \(K\)-embedding of \(H\)-fields
  \[
    \iota:L\hookrightarrow M.
  \]
  Since \(\iota\) fixes \(K\) pointwise and preserves the \(H\)-action, one has
  \[
    M\models\exists x\,\varphi(x,a).
  \]
  We may regard \(\varphi(x,a)\) as an \(L_G\)-formula, since \(L_H\) is a sublanguage of \(L_G\).
  As \(M/K\) is regular and \(K\) is pseudo-existentially closed as a \(G\)-field, it follows that
  \[
    K\models\exists x\,\varphi(x,a).
  \]
  Therefore \(K\), with its action restricted to \(H\), is a pseudo-existentially closed \(H\)-field.

  Finally, every existentially closed \(G\)-field is pseudo-existentially closed as a \(G\)-field, since every regular \(G\)-field extension is, in particular, a \(G\)-field extension.
  The final assertion therefore follows from the first part.
\end{proof}

\subsection{Quotient fields and inflation}

If \(\pi:\Gamma\twoheadrightarrow Q\) has kernel \(D\), then inflation along \(\pi\) identifies \(Q\)-fields with \(\Gamma\)-fields on which every element of \(D\) acts trivially.
In particular, being p.e.c. in the latter class is not a new relative notion: it is precisely being p.e.c. as a \(Q\)-field (equivalently, as a \(\Gamma/D\)-field).

\begin{lemma}\label{lemma:Inflation}
  Let \(\pi:\Gamma\twoheadrightarrow Q\) be a surjective group homomorphism and let \(K\) be a \(Q\)-field.
  Let \(K_\pi\) denote the \(\Gamma\)-field with the same underlying field and action
  \[
    \gamma\cdot a:=\pi(\gamma)(a).
  \]
  Then:
  \begin{enumerate}[label=\textup{(\roman*)}]
    \item \(K\) is \(\aleph_0\)-saturated as a \(Q\)-field if and only if \(K_\pi\) is \(\aleph_0\)-saturated as a \(\Gamma\)-field;
    \item \(K\) is p.e.c. as a \(Q\)-field if and only if \(K_\pi\) is existentially closed in every regular \(\Gamma\)-field extension of \(K_\pi\) on which \(\ker(\pi)\) acts trivially.
  \end{enumerate}
\end{lemma}

\begin{proof}
  On every \(\Gamma\)-field on which \(\ker(\pi)\) acts trivially, the action factors uniquely through \(\Gamma/\ker(\pi)\cong Q\).
  Conversely, every \(Q\)-field action inflates along \(\pi\).
  These operations are mutually inverse and do not change the underlying field, so they preserve regular extensions.

  Every formula in the \(\Gamma\)-field language translates to the \(Q\)-field formula obtained by replacing each symbol \(\gamma\) by \(\pi(\gamma)\).
  Conversely, fix a set-theoretic section \(s:Q\to\Gamma\) of \(\pi\), and replace each action symbol indexed by \(a\in Q\) by the symbol indexed by \(s(a)\).
  For the corresponding structures \(K\) and \(K_\pi\), these translations induce mutually inverse correspondences between complete types: in \(\operatorname{Th}(K_\pi)\), the symbols indexed by \(\gamma\) and \(s(\pi(\gamma))\) define the same automorphism.
  They also preserve existential formulas.
  This proves both assertions.
\end{proof}

\section{Cocycles on Baumslag--Solitar Groups}\label{sec:BS-cocycles}

Replacing \((m,n)\) by \((-m,-n)\) does not change the group, so we may and do assume \(m>0\).
Put
\[
  r:=\frac nm\in\Q^\times.
\]

\subsection{The \(\tau\)-exponent-sum homomorphism \(\ell\) and the 1-cocycle \(\beta\)}

The exponent sum in \(\tau\) defines a homomorphism
\[
  \ell:\BS(m,n)\longrightarrow\Z, \quad \ell(\sigma) = 0, \quad \ell(\tau) = 1.
\]
We regard the additive group
\[
  \Z\left[\frac1{mn}\right]
\]
as a \(\BS(m,n)\)-module by
\[
  \gamma\cdot a := r^{\ell(\gamma)}a.
\]
This is well-defined because both \(r=n/m\) and \(r^{-1}=m/n\) are units in \(\Z[1/(mn)]\).

\begin{proposition}\label{prop:beta}
  There is a unique 1-cocycle
  \[
    \beta:\BS(m,n)\longrightarrow\Z\left[\frac1{mn}\right]
  \]
  such that
  \[
    \beta(\sigma)=1,\quad \beta(\tau)=0,
  \]
  and
  \begin{equation}\label{eq:beta-cocycle}
    \beta(\gamma\delta) = \beta(\gamma)+r^{\ell(\gamma)}\beta(\delta)\quad (\gamma,\delta\in\BS(m,n)).
  \end{equation}
\end{proposition}

\begin{proof}
  On the free group on \(\sigma,\tau\), the prescribed values determine a unique 1-cocycle for the module structure above.
  It remains only to check the defining relation.
  Since \(\beta(\tau^{-1})=0\), one has
  \[
    \beta(\tau\sigma^m\tau^{-1})=r\,\beta(\sigma^m)=rm=n=\beta(\sigma^n).
  \]
  Hence the 1-cocycle descends to \(\BS(m,n)\).
  Uniqueness follows from the generators.
\end{proof}

\begin{lemma}\label{lem:reduction}
  If \(N\geq1\) is coprime to \(mn\), there is a canonical surjective homomorphism of additive groups
  \[
    \operatorname{red}_N:\Z\left[\frac1{mn}\right]\longrightarrow\Z/N\Z
  \]
  whose kernel is
  \[
    N\Z\left[\frac1{mn}\right].
  \]
  Moreover, if \(S\subseteq\Z[1/(mn)]\) is finite, then \(\operatorname{red}_N\) is injective on \(S\) for all but finitely many primes \(N\) coprime to \(mn\).
\end{lemma}

\begin{proof}
  Since \(mn\) is a unit modulo \(N\), reduction extends uniquely from \(\Z\) to the localization \(\Z[1/(mn)]\), and clearing a denominator proves the kernel statement.
  For the final assertion, write the finitely many non-zero differences of elements of \(S\) with denominators supported on the prime divisors of \(mn\).
  Reduction can identify two elements of \(S\) only at a prime dividing one of these numerators, so only finitely many primes must be excluded.
\end{proof}

\subsection{The cyclotomic condition}

\begin{definition}\label{def:admissible}
  Let \(K\) be a \(\BS(m,n)\)-field.
  A prime \(q\) is
  \emph{admissible for \(\BS(m,n)\) over \(K\)} if it is coprime to \(mn\)
  and to the characteristic exponent of \(K\).
\end{definition}

Let \(K\) be a \(\BS(m,n)\)-field and let \(q\) be an admissible prime.
Put
\[
  w:=nm^{-1}\in(\Z/q\Z)^\times,
\]
and also write \(w\in\{1,\ldots,q-1\}\) for its standard integer representative.
Whenever a residue class modulo \(q\) occurs as an exponent of a \(q\)-th root of unity, the exponent is understood modulo \(q\).
For a variable \(z\), let
\[
  \theta_q(z):\equiv z^q=1\wedge z\neq1\wedge\sigma(z)=z\wedge\tau(z)=z^w.
\]

\begin{definition}\label{def:cyclotomic-condition}
  We say that \(K\) \emph{satisfies the cyclotomic condition \(\theta_q\)} if it contains a primitive \(q\)-th root of unity \(\zeta\) satisfying \(\theta_q(\zeta)\).
  More generally, if
  \[
    \pi:\BS(m,n)\twoheadrightarrow Q
  \]
  is a surjective homomorphism, a \(Q\)-field \emph{satisfies \(\theta_q\) relative to \(\pi\)} if \(q\) is admissible for \(\BS(m,n)\) over its underlying field and it contains a primitive \(q\)-th root of unity satisfying the same equations, with \(\sigma,\tau\) interpreted through \(\pi\).
\end{definition}

If \(q=2\) is admissible, then \(mn\) is odd and \(\operatorname{char}K\neq2\).
In this case \(w=1\), the unique primitive second root of unity is \(-1\), and \(\theta_2\) is automatic.
Thus for \(\BS(1,1)\cong\Z^2\) in characteristic different from \(2\), the cyclotomic condition imposes no additional hypothesis, and the argument for ambient groups uses \(\nu=q-1=1\).

\subsection{A 1-cocycle}

For a field \(K\), write
\[
  \mu_q(K):=\{u\in K^\times:u^q=1\}.
\]

\begin{lemma}\label{lem:oneCocycle}
  Let \(K\) be a \(\BS(m,n)\)-field, let \(q\) be admissible for \(\BS(m,n)\) over \(K\), and let \(\zeta\in K\) satisfy \(\theta_q(\zeta)\).
  Define
  \[
    \chi:\BS(m,n)\longrightarrow\mu_q(K),\quad \chi(\gamma):=\zeta^{\operatorname{red}_q(\beta(\gamma))}.
  \]
  Then \(\chi\) is a \(1\)-cocycle.
  Moreover,
  \[
    \chi(\sigma)=\zeta,\quad \chi(\tau)=1,\quad \chi|_{\ker\beta}=1.
  \]
\end{lemma}

\begin{proof}
  For \(\gamma,\delta\in\BS(m,n)\), Equation~\eqref{eq:beta-cocycle} gives
  \[
    \operatorname{red}_q(\beta(\gamma\delta))=\operatorname{red}_q(\beta(\gamma))+w^{\ell(\gamma)}\operatorname{red}_q(\beta(\delta)).
  \]
  Since
  \[
    \gamma(\zeta)=\zeta^{w^{\ell(\gamma)}},
  \]
  it follows that
  \[
    \chi(\gamma\delta)=\chi(\gamma)\,\gamma\bigl(\chi(\delta)\bigr).
  \]
  The final assertions follow immediately from \(\beta(\sigma)=1\), \(\beta(\tau)=0\), and the definition of \(\chi\).
\end{proof}

\section{1-Cocycle Realization}\label{sec:cocycle}

The construction in this section is a \(K^\times\)-valued analogue of the standard induction/coinduction construction for a subgroup and the associated Eckmann--Shapiro formalism; compare \cite[Chapter~III, \S\S5--6]{Brown1982}.
We give the explicit formula because we need an explicit induced action realizing a prescribed \(1\)-cocycle while keeping a specified normal subgroup acting trivially.

\begin{lemma}\label{lem:cocycleRealization}
  Let \(H\leq G\), let \(K\) be a \(G\)-field, and let \(\chi:H\to K^\times\) satisfy the \(1\)-cocycle condition
  \[
    \chi(h_1h_2)=\chi(h_1)\,h_1\bigl(\chi(h_2)\bigr)\quad(h_1,h_2\in H).
  \]
  Choose a section
  \[
    \rho:G/H\longrightarrow G,\quad \rho(H)=1,
  \]
  and put
  \[
    \alpha(\gamma,\lambda):=\rho(\gamma\lambda)^{-1}\gamma\rho(\lambda)\quad(\gamma\in G,\ \lambda\in G/H).
  \]
  If the elements \(z_\lambda\), \(\lambda\in G/H\), are algebraically independent over \(K\), then the formulas
  \begin{equation}\label{eq:cochain-action}
    \gamma(z_\lambda):=\rho(\gamma\lambda)\bigl(\chi(\alpha(\gamma,\lambda))\bigr)z_{\gamma\lambda}
  \end{equation}
  extend the given action on \(K\) to a \(G\)-field structure on the purely transcendental extension
  \[
    E:=K\bigl(z_\lambda:\lambda\in G/H\bigr).
  \]
  In particular,
  \[
    h(z_H)=\chi(h)z_H\quad(h\in H).
  \]

  Moreover, let \(D\trianglelefteq G\) satisfy \(D\subseteq H\) and \(\chi|_D=1\).
  If \(D\) acts trivially on \(K\), then it acts trivially on \(E\).
\end{lemma}

\begin{proof}
  For \(\gamma,\delta\in G\) and \(\lambda\in G/H\), one has
  \[
    \gamma\rho(\lambda)=\rho(\gamma\lambda)\alpha(\gamma,\lambda)
  \]
  and
  \[
    \alpha(\gamma\delta,\lambda)=\alpha(\gamma,\delta\lambda)\alpha(\delta,\lambda).
  \]
  Using these identities together with the \(1\)-cocycle identity for \(\chi\) gives
  \[
    \gamma\bigl(\delta(z_\lambda)\bigr)=(\gamma\delta)(z_\lambda),
  \]
  so \eqref{eq:cochain-action} defines a \(G\)-action on \(E\).
  Since \(hH=H\), \(\rho(H)=1\), and \(\alpha(h,H)=h\) for \(h\in H\), one gets
  \[
    h(z_H)=\chi(h)z_H.
  \]

  Finally, suppose that \(D\trianglelefteq G\), \(D\subseteq H\), and \(\chi|_D=1\).
  For \(d\in D\) and \(\lambda\in G/H\), normality gives \(d\lambda=\lambda\) and
  \[
    \alpha(d,\lambda)=\rho(\lambda)^{-1}d\rho(\lambda)\in D.
  \]
  Hence \eqref{eq:cochain-action} gives \(d(z_\lambda)=z_\lambda\).
  If \(D\) acts trivially on \(K\), it therefore acts trivially on all of \(E\).
\end{proof}

\begin{corollary}\label{cor:globalCocycleRealization}
  Let \(K\) be a \(G\)-field and let \(\chi:G\to K^\times\) be a \(1\)-cocycle.
  If \(z\) is transcendental over \(K\), then the formula
  \[
    \gamma(z)=\chi(\gamma)z\quad(\gamma\in G)
  \]
  extends the action on \(K\) to a \(G\)-field structure on \(K(z)\).
  If \(D\trianglelefteq G\), \(\chi|_D=1\), and \(D\) acts trivially on \(K\), then \(D\) acts trivially on \(K(z)\).
\end{corollary}

\begin{proof}
  Apply Lemma~\ref{lem:cocycleRealization} with \(H=G\).
  Then \(G/H\) has one element, so the construction reduces to the displayed formula.
\end{proof}

\section{The Realization Construction}\label{sec:realization-construction}
Fix a prime \(q\); we keep this prime fixed throughout Sections~\ref{sec:realization-construction}--\ref{sec:BS-obstruction}.
Whenever a base field satisfying \(\theta_q\) is assumed, \(q\) is admissible by Definition~\ref{def:cyclotomic-condition}.

\subsection{The realization problem}

Let \(K\) be a \(\BS(m,n)\)-field and let \(c\in K\).
Let \(\mathcal F=\langle\sigma,\tau\rangle\) be the free group and let
\[
  \pi:\mathcal F\twoheadrightarrow\BS(m,n)
\]
be the natural map.
Regard \(K\) as an \(\mathcal F\)-field through \(\pi\), and let
\[
  \widetilde B_\bullet(c):\mathcal F\longrightarrow K
\]
be the unique \(1\)-cocycle such that
\[
  \widetilde B_\sigma(c)=c,\quad \widetilde B_\tau(c)=0.
\]
Fix a set-theoretic section \(s:\BS(m,n)\to\mathcal F\) such that \(s(\sigma^j)=\sigma^j\) for \(j\in\Z\) and \(s(\tau)=\tau\), and put
\[
  B_\gamma(c):=\widetilde B_{s(\gamma)}(c)\quad(\gamma\in\BS(m,n)).
\]
Thus every \(B_\gamma(c)\) is a fixed finite \(\BS(m,n)\)-field term in \(c\).

Since the \(\mathcal F\)-action on \(K\) factors through \(\BS(m,n)\), the cocycle \(\widetilde B_\bullet(c)\) descends to \(\BS(m,n)\) if and only if
\[
  \tau(B_{\sigma^m}(c))=B_{\sigma^n}(c).
\]
In this case their values are independent of the chosen representatives and give the unique map
\[
  B_\bullet(c):\BS(m,n)\longrightarrow K
\]
satisfying
\[
  B_1(c)=0,\quad B_\sigma(c)=c,\quad B_\tau(c)=0,
\]
and
\[
  B_{\gamma\delta}(c)=B_\gamma(c)+\gamma\bigl(B_\delta(c)\bigr)\quad(\gamma,\delta\in\BS(m,n)).
\]

Suppose that \(L/K\) is a \(\BS(m,n)\)-field extension and that \(x\in L\) satisfies \(\sigma(x)=x+c\) and \(\tau(x)=x\).
Induction on words in \(\mathcal F\) gives
\[
  u(x)-x=\widetilde B_u(c)\quad(u\in\mathcal F).
\]
Hence \(\widetilde B_\bullet(c)\) descends to \(\BS(m,n)\), and
\[
  \gamma(x)=x+B_\gamma(c)\quad(\gamma\in\BS(m,n)).
\]
Conversely, if the displayed descent condition holds and \(x\) is transcendental over \(K\), then the formulas \(\gamma(x)=x+B_\gamma(c)\) define a \(\BS(m,n)\)-action on \(K(x)\) extending the action on \(K\).

\begin{remark}
  If \(\BS(m,n)\) acts trivially on \(K\), the descent condition is
  \[
    mc=nc.
  \]
  Thus, if \(\operatorname{char}K=0\) and \(m\neq n\), it fails for \(c=1\); in particular, no such cocycle exists for the trivial \(\BS(1,2)\)-action with \(c=1\).
\end{remark}

When \(m=n=1\) and \(j>0\), one has
\[
  B_{\sigma^j}(c)=c+\sigma(c)+\cdots+\sigma^{j-1}(c),
\]
the iterated sum appearing in Kikyo's presentation of Hrushovski's argument \cite{Kikyo2004}.

\begin{definition}\label{def:realizationProblem}
  The \emph{realization problem} \(\psi(x,y,c,e,g)\) is given by
  \[
    \psi(x,y,c,e,g) :\equiv x\neq0\wedge\sigma(x)=x+c\wedge\tau(x)=x\wedge y^q=ex\wedge\tau(y)=gy.
  \]
\end{definition}

Hrushovski's realization problem for two commuting automorphisms, as presented by Kikyo \cite{Kikyo2004}, is
\[
  \sigma(x)=x+c,\quad \tau(x)=x+c,\quad y^3=x,\quad \tau(y)=\zeta\sigma(y),
\]
where \(\zeta\) is a primitive cube root of unity.
Definition~\ref{def:realizationProblem} replaces the simultaneous translations by \(\tau(x)=x\), the cubic cover by \(y^q=ex\), and the fixed twist \(\zeta\sigma(y)\) by \(\tau(y)=gy\).
Most importantly, the root of unity \(\zeta\) no longer occurs in the realization problem itself; it enters only through the cyclotomic condition \(\theta_q\) and the auxiliary cocycle \(\chi\).

\subsection{Regular realization}\label{subsec:regular-realization}
We introduce the following partial type, which gives an explicit sufficient condition for solving the realization problem.
When \(k<0\), an equation of the form \(u=g^k v\), in a context where \(g\neq0\), is an abbreviation for the equivalent ring-language equation \(g^{-k}u=v\).
Thus the formulas below are genuine first-order formulas in the language of fields with the group action.

\begin{definition}\label{def:Sigma}
  Let \(\Sigma(c,e,g)\) be the partial type consisting of
  \begin{enumerate}
    \item \(\tau(B_{\sigma^m}(c))=B_{\sigma^n}(c)\);
    \item \(eg\neq0\);
    \item for every \(\gamma\in\ker\beta\), the formulas \(\gamma(g)=g\) and \(\gamma(e)=g^{q\ell(\gamma)}e\);
    \item all formulas in \(\left\{ B_\gamma(c)=B_\delta(c) \colon \beta(\gamma)=\beta(\delta) \right\} \cup \left\{ B_\gamma(c)\neq B_\delta(c) \colon \beta(\gamma)\neq\beta(\delta) \right\}\);
  \end{enumerate}
\end{definition}

\begin{proposition}\label{prop:SigmaRealization}
  Suppose that \(K\) is a \(\BS(m,n)\)-field with \(\operatorname{char}K\neq q\) and that \(c,e,g\in K\) realize the partial type \(\Sigma(c,e,g)\).
  Then there exists a \(\BS(m,n)\)-field extension \(L/K\) such that \(L/K\) is regular as an extension of pure fields and \(L\) contains a solution of the realization problem \(\psi(x,y,c,e,g)\).
\end{proposition}

\begin{proof}
  Put \(G:=\BS(m,n)\) and \(H:=\ker\beta\).
  The cocycle identity shows that \(H\) is a subgroup of \(G\).
  Moreover, for \(\gamma,\delta\in G\),
  \[
    \beta(\gamma^{-1}\delta)=r^{-\ell(\gamma)}\bigl(\beta(\delta)-\beta(\gamma)\bigr).
  \]
  Consequently,
  \[
    \beta(\gamma)=\beta(\delta)\quad\Longleftrightarrow\quad \gamma^{-1}\delta\in H\quad\Longleftrightarrow\quad \gamma H=\delta H.
  \]
  Condition~\textup{(1)} of \(\Sigma\) says exactly that the assignments \(\sigma(x)=x+c\) and \(\tau(x)=x\) respect the Baumslag--Solitar relation, so they extend the \(G\)-action to
  \[
    F:=K(x).
  \]
  By the preceding discussion, this action satisfies
  \[
    \gamma(x)=x+B_\gamma(c).
  \]
  The equivalence above and condition~\textup{(4)} give
  \[
    \Stab_G(x)=H.
  \]
  Define
  \[
    \chi:H\longrightarrow K^\times,\quad \chi(h):=g^{\ell(h)}.
  \]
  By condition~\textup{(3)}, every element of \(H\) fixes \(g\), so the homomorphism property of \(\ell\) shows that \(\chi\) is a \(1\)-cocycle.
  Choose a section \(\rho:G/H\to G\) with \(\rho(H)=1\), and put
  \[
    \alpha(\gamma,\lambda):=\rho(\gamma\lambda)^{-1}\gamma\rho(\lambda).
  \]
  By Lemma~\ref{lem:cocycleRealization}, for algebraically independent elements \(Z_\lambda\), \(\lambda\in G/H\), the formulas
  \begin{equation}\label{eq:Z-action}
    \gamma(Z_\lambda):=\rho(\gamma\lambda)\bigl(\chi(\alpha(\gamma,\lambda))\bigr)Z_{\gamma\lambda}
  \end{equation}
  define a \(G\)-action on
  \[
    E:=F(Z_\lambda:\lambda\in G/H).
  \]

  Put \(W_0:=ex\).
  For every \(h\in H\), condition~\textup{(3)} of \(\Sigma\) and the equality \(h(x)=x\) give
  \[
    h(W_0)=h(e)h(x)=g^{q\ell(h)}ex=\chi(h)^qW_0.
  \]
  For each \(\lambda\in G/H\), put
  \[
    W_\lambda:=\rho(\lambda)(W_0).
  \]
  Let
  \[
    R:=F[Z_\lambda:\lambda\in G/H]
  \]
  and consider the ideal
  \[
    I:=\left\langle Z_\lambda^q-W_\lambda:\lambda\in G/H\right\rangle.
  \]

  \begin{claim}
    The ideal \(I\) is \(G\)-invariant.
  \end{claim}

  \begin{proof}
    For \(\gamma\in G\), \(\lambda\in G/H\), and \(h:=\alpha(\gamma,\lambda)\), Equation~\eqref{eq:Z-action} and the identity \(h(W_0)=\chi(h)^qW_0\) give
    \[
      \begin{aligned}
        \gamma(W_\lambda)
         & =\rho(\gamma\lambda)\bigl(\chi(h)\bigr)^qW_{\gamma\lambda},   \\
        \gamma(Z_\lambda)^q
         & =\rho(\gamma\lambda)\bigl(\chi(h)\bigr)^qZ_{\gamma\lambda}^q.
      \end{aligned}
    \]
    Hence
    \[
      \gamma(Z_\lambda^q-W_\lambda)=\rho(\gamma\lambda)\bigl(\chi(h)\bigr)^q\bigl(Z_{\gamma\lambda}^q-W_{\gamma\lambda}\bigr)\in I.
    \]
  \end{proof}

  \begin{claim}
    Both \(R/I\) and \((R/I)\otimes_K\overline K\) are integral domains.
  \end{claim}

  \begin{proof}
    Fix an algebraic closure \(\overline K\) of \(K\).
    For every \(\lambda\in G/H\),
    \[
      W_\lambda=\rho(\lambda)(e)\bigl(x+B_{\rho(\lambda)}(c)\bigr).
    \]
    If \(\lambda\neq\kappa\), then \(\rho(\lambda)H\neq\rho(\kappa)H\).
    By the equivalence above,
    \[
      \beta(\rho(\lambda))\neq\beta(\rho(\kappa)).
    \]
    Condition~\textup{(4)} of \(\Sigma\) therefore gives
    \[
      B_{\rho(\lambda)}(c)\neq B_{\rho(\kappa)}(c).
    \]
    Thus the \(W_\lambda\) have pairwise distinct linear prime divisors
    \[
      x+B_{\rho(\lambda)}(c)
    \]
    in \(\overline K[x]\).

    We first show that their classes are linearly independent over \(\mathbb F_q\) in
    \[
      \overline K(x)^\times/\overline K(x)^{\times q}.
    \]
    Let \(\lambda_1,\ldots,\lambda_k\) be distinct and suppose
    \[
      W_{\lambda_1}^{d_1}\cdots W_{\lambda_k}^{d_k}\in\overline K(x)^{\times q},\quad 0\leq d_i<q.
    \]
    Let \(v_i\) be the discrete valuation corresponding to
    \[
      x+B_{\rho(\lambda_i)}(c).
    \]
    Then
    \[
      v_i(W_{\lambda_j})=\delta_{ij},
    \]
    so the valuation of the displayed product is \(d_i\).
    Since valuations of \(q\)-th powers are divisible by \(q\), every \(d_i\) is zero.

    Now let \(U\subseteq G/H\) be finite and put
    \[
      C:=K(x)\otimes_K\overline K.
    \]
    The extension \(K(x)/K\) is regular, so \(C\) is an integral domain and embeds naturally in \(\overline K(x)\).
    Its fraction field is \(\overline K(x)\).
    Set
    \[
      C_U:=C[Z_\lambda:\lambda\in U]/\langle Z_\lambda^q-W_\lambda:\lambda\in U\rangle.
    \]
    Because the defining polynomials are monic, \(C_U\) is a free \(C\)-module with basis
    \[
      \prod_{\lambda\in U}Z_\lambda^{d_\lambda},\quad 0\leq d_\lambda<q.
    \]
    In particular it is torsion-free over \(C\), so localization gives an injection
    \[
      C_U\hookrightarrow\overline K(x)[Z_\lambda:\lambda\in U]/\langle Z_\lambda^q-W_\lambda:\lambda\in U\rangle.
    \]
    The valuation argument above says precisely that the classes of the \(W_\lambda\), \(\lambda\in U\), are independent in \(\overline K(x)^\times/\overline K(x)^{\times q}\).
    The field \(\overline K(x)\) contains the \(q\)-th roots of unity, and its characteristic is different from \(q\) by hypothesis.
    Kummer theory therefore shows that adjoining \(q\)-th roots of the \(W_\lambda\) gives an extension of degree \(q^{|U|}\); see \cite[Theorem~5.30]{Milne2022}.
    The ring on the right has \(\overline K(x)\)-dimension \(q^{|U|}\).
    Its natural map onto the field obtained by adjoining chosen roots is therefore an isomorphism.
    Hence \(C_U\) is an integral domain.

    If \(U\subseteq V\) are finite, the same monic-basis argument shows that \(C_V\) is a free \(C_U\)-module.
    In particular, the natural map \(C_U\to C_V\) is injective.
    The ring \((R/I)\otimes_K\overline K\) is therefore the directed union of the domains \(C_U\), and hence is an integral domain.
    Since \(\overline K/K\) is faithfully flat, the natural map
    \[
      R/I\longrightarrow (R/I)\otimes_K\overline K
    \]
    is injective, and \(R/I\) is an integral domain as well.
  \end{proof}

  Put
  \[
    \overline R:=R/I
  \]
  and
  \[
    L:=\Frac(\overline R).
  \]
  Since \(I\) is \(G\)-invariant, the \(G\)-action on \(R\) descends to \(\overline R\) and extends uniquely to a \(G\)-action on \(L\).
  Moreover,
  \[
    L\otimes_K\overline K\cong(\overline R\otimes_K\overline K)_{\overline R\setminus\{0\}}.
  \]
  Since \(\overline R\otimes_K\overline K\) is an integral domain and the image of every element of \(\overline R\setminus\{0\}\) is nonzero, the ring \(L\otimes_K\overline K\) is an integral domain.
  By the standard tensor-product criterion for regular extensions recalled above (see also \cite[Chapter~3]{FriedJarden2023}), \(L/K\) is regular as an extension of pure fields.
  Let \(y\in L\) be the image of \(Z_H\).
  Since \(\rho(H)=1\), one has
  \[
    y^q=W_H=W_0=ex.
  \]
  For every \(h\in H\), one has \(\alpha(h,H)=h\).
  Therefore
  \[
    h(y)=\chi(h)y=g^{\ell(h)}y.
  \]
  Since \(\tau\in H\) and \(\ell(\tau)=1\), it follows that
  \[
    \tau(y)=gy.
  \]
  Finally, \(x\) is transcendental over \(K\), and hence \(x\neq0\).
  Consequently,
  \[
    L\models\psi(x,y,c,e,g).
  \]
\end{proof}

The preceding construction also gives the following refinement.

\begin{proposition}\label{prop:regularRealizationTrivialD}
  Under the hypotheses of Proposition~\ref{prop:SigmaRealization}, let \(D\trianglelefteq\BS(m,n)\) satisfy
  \[
    D\subseteq\ker\beta\cap\ker\ell.
  \]
  If \(D\) acts trivially on \(K\), then the regular extension \(L/K\) may be chosen so that \(D\) acts trivially on \(L\).
\end{proposition}

\begin{proof}
  Let \(H=\ker\beta\).
  For \(d\in D\), condition~\textup{(4)} of \(\Sigma\) gives
  \[
    B_d(c)=B_1(c)=0,
  \]
  so \(d(x)=x\).
  Thus \(D\) acts trivially on \(F=K(x)\).
  The cocycle used in Proposition~\ref{prop:SigmaRealization} is
  \[
    \chi(h)=g^{\ell(h)}\quad(h\in H).
  \]
  Since \(D\subseteq H\) and \(\ell(D)=0\), one has \(\chi|_D=1\).
  The ``Moreover'' part of Lemma~\ref{lem:cocycleRealization} therefore makes every variable \(Z_\lambda\) fixed by \(D\).
  Thus \(D\) acts trivially on \(R\), and hence also on \(\overline R=R/I\) and its fraction field.
\end{proof}

\section{The Baumslag--Solitar Obstruction}\label{sec:BS-obstruction}

\subsection{Obstruction to the realization problem}

The following proposition gives an obstruction to solving \(\psi(x,y,c,e,g)\) when the \(\sigma\)-orbit of \(x\) is periodic.

\begin{proposition}\label{prop:realizationObstruction}
  Let \(N\geq1\).
  Suppose that \(c,e,g,\zeta\in K\) satisfy
  \[
    K\models\theta_q(\zeta)\wedge eg\neq0\wedge B_{\sigma^N}(c)=0\wedge \sigma^N(e)=e\wedge\sigma^{nN}(g)\neq g.
  \]
  Then no \(\BS(m,n)\)-field extension \(L/K\) contains a solution of \(\psi(x,y,c,e,g)\).
\end{proposition}

\begin{proof}
  Suppose that \(x,y\in L\) satisfy \(\psi(x,y,c,e,g)\).
  Since \(B_{\sigma^N}(c)=0\), one has \(\sigma^N(x)=x\).
  Applying \(\sigma^N\) to \(y^q=ex\) and using \(\sigma^N(e)=e\) gives
  \[
    \sigma^N(y)^q=y^q.
  \]
  As \(y\neq0\), put
  \[
    \xi:=\frac{\sigma^N(y)}y\in\mu_q(L).
  \]
  Because \(\zeta\in K\) is a primitive \(q\)-th root of unity, every \(q\)-th root of unity in \(L\) already lies in \(K\).
  Hence \(\xi\in\mu_q(K)=\langle\zeta\rangle\); in particular \(\sigma(\xi)=\xi\) and \(\tau(\xi)=\xi^w\).
  Hence
  \[
    \sigma^{jN}(y)=\xi^j y\quad(j\in\Z).
  \]
  The Baumslag--Solitar relation gives
  \[
    \tau\sigma^{mN}=\sigma^{nN}\tau.
  \]
  Applying this to \(y\), the two sides are
  \[
    \tau\sigma^{mN}(y)=\xi^{wm}gy
  \]
  and
  \[
    \sigma^{nN}\tau(y)=\sigma^{nN}(g)\xi^n y.
  \]
  Since \(wm\equiv n\pmod q\), one has \(\xi^{wm}=\xi^n\), and cancellation yields
  \[
    \sigma^{nN}(g)=g,
  \]
  contrary to the hypothesis.
\end{proof}

\subsection{The obstruction theorem}

\begin{theorem}\label{th:BSQuotient}
  Let
  \[
    \varphi:\BS(m,n)\twoheadrightarrow Q
  \]
  be a surjective homomorphism such that
  \[
    \ker\varphi\subseteq\ker\beta\cap\ker\ell.
  \]
  Then there is no \(\aleph_0\)-saturated p.e.c. \(Q\)-field satisfying \(\theta_q\) relative to \(\varphi\).
\end{theorem}

\begin{proof}
  Put \(D:=\ker\varphi\), and suppose toward a contradiction that \(K\) is such a \(Q\)-field.
  Inflate its action along \(\varphi\), and continue to denote the resulting \(\BS(m,n)\)-field by \(K\).
  By Lemma~\ref{lemma:Inflation}, it is \(\aleph_0\)-saturated, and every existential formula realized in a regular \(\BS(m,n)\)-field extension on which \(D\) acts trivially is already realized in \(K\).
  Fix a primitive \(q\)-th root of unity \(\zeta\in K\) satisfying \(\theta_q(\zeta)\).
  Let
  \[
    p(c,e,g):=\Sigma(c,e,g)\cup\{\neg\exists x\exists y\,\psi(x,y,c,e,g)\}.
  \]
  We show that \(p\) is finitely satisfiable in \(K\).

  Let \(p_0\subseteq_{\mathrm{fin}}p\).
  Enlarging \(p_0\) if necessary, we may assume that \(eg\neq0\) belongs to \(p_0\).
  First construct \(g\) and \(e\), independently of the finite periodic orbit.
  Let \(\chi\) be the \(1\)-cocycle from Lemma~\ref{lem:oneCocycle}.
  Since \(D\subseteq\ker\beta\), one has \(\chi|_D=1\).
  Applying Corollary~\ref{cor:globalCocycleRealization}, we obtain a purely transcendental \(\BS(m,n)\)-field extension \(K_1/K\) containing \(g\neq0\) such that \(D\) acts trivially on \(K_1\) and
  \[
    \gamma(g)=\chi(\gamma)g\quad(\gamma\in\BS(m,n)).
  \]
  In particular,
  \[
    \sigma(g)=\zeta g,\quad \tau(g)=g,\quad \gamma(g)=g\quad(\gamma\in\ker\beta).
  \]
  Since \(\chi(\gamma)^q=1\), the element \(g^q\) is fixed by all of \(\BS(m,n)\).
  Hence
  \[
    \eta:\BS(m,n)\longrightarrow K_1^\times,\quad \eta(\gamma):=g^{q\ell(\gamma)}
  \]
  is a \(1\)-cocycle.
  Moreover, \(\eta|_D=1\) because \(D\subseteq\ker\ell\).
  A second application of Corollary~\ref{cor:globalCocycleRealization} gives a purely transcendental extension \(K_2/K_1\) containing \(e\neq0\) such that \(D\) acts trivially on \(K_2\) and
  \[
    \gamma(e)=g^{q\ell(\gamma)}e\quad(\gamma\in\BS(m,n)).
  \]
  Thus
  \[
    \sigma(e)=e,\quad \tau(e)=g^q e,
  \]
  and condition~\textup{(3)} of \(\Sigma\) holds in full.

  Only now choose the length of the finite periodic orbit.
  Let \(S\subseteq\operatorname{Im}\beta\) be the finite set of \(\beta\)-values appearing in the finitely many instances of condition~\textup{(4)} of \(\Sigma\) occurring in \(p_0\).
  By Lemma~\ref{lem:reduction}, choose a prime \(N\neq q\), coprime to \(mn\), such that \(\operatorname{red}_N\) is injective on \(S\).
  Since \(q\nmid nN\), we have
  \[
    \sigma^{nN}(g)=\zeta^{nN}g\neq g,
  \]
  while \(\sigma^N(e)=e\).

  It remains to construct \(c\).
  Put
  \[
    w_N:=nm^{-1}\in(\Z/N\Z)^\times.
  \]
  Adjoin algebraically independent variables
  \[
    X_j\quad(j\in\Z/N\Z)
  \]
  and define
  \[
    \sigma(X_j)=X_{j+1},\quad \tau(X_j)=X_{w_Nj}.
  \]
  Since \(w_Nm\equiv n\pmod N\), these permutations satisfy the Baumslag--Solitar relation.
  More precisely, induction from the 1-cocycle identity for \(\beta\) gives
  \[
    \gamma(X_j)=X_{\operatorname{red}_N(\beta(\gamma))+w_N^{\ell(\gamma)}j}\quad(\gamma\in\BS(m,n)).
  \]
  Hence every element of \(D\subseteq\ker\beta\cap\ker\ell\) fixes every \(X_j\).
  Put
  \[
    X:=X_0,\quad c:=\sigma(X)-X=X_1-X_0.
  \]
  Then
  \[
    B_{\sigma^j}(c)=X_{\operatorname{red}_N(j)}-X_0\quad(j\in\Z),\quad B_{\sigma^N}(c)=0,\quad \tau(B_{\sigma^m}(c))=B_{\sigma^n}(c),
  \]
  and for every \(\gamma\in\BS(m,n)\),
  \[
    B_\gamma(c)=X_{\operatorname{red}_N(\beta(\gamma))}-X_0.
  \]
  By the choice of \(N\), all finitely many equalities and inequalities from condition~\textup{(4)} of \(\Sigma\) occurring in \(p_0\) are satisfied.

  Let \(K'/K\) be obtained by adjoining \(g\), then \(e\), and then the finite family \((X_j)_{j\in\Z/N\Z}\).
  This is a regular \(\BS(m,n)\)-field extension on which \(D\) acts trivially.
  In \(K'\), the tuple \((c,e,g)\) satisfies every formula in \(p_0\cap\Sigma(c,e,g)\), together with
  \[
    B_{\sigma^N}(c)=0,\quad \sigma^N(e)=e,\quad \sigma^{nN}(g)\neq g.
  \]
  Since \(D\) acts trivially on \(K'\), its action factors through \(Q\), making \(K'/K\) a regular \(Q\)-field extension.
  As \(K\) is p.e.c. as a \(Q\)-field, the same finite conjunction is realized by some \((c_0,e_0,g_0)\in K\).
  Proposition~\ref{prop:realizationObstruction} then shows that no \(\BS(m,n)\)-field extension of \(K\) can realize \(\psi(x,y,c_0,e_0,g_0)\).
  Hence \(p_0\) is realized in \(K\).

  By the Compactness Theorem, the finite satisfiability just proved shows that \(p\) is consistent with \(\operatorname{Th}(K)\).
  Extend \(p\) to a complete type over the empty set.
  By \(\aleph_0\)-saturation, some \((c,e,g)\in K\) realizes \(p\).
  Then \(\Sigma(c,e,g)\) holds in \(K\), while
  \[
    K\models\neg\exists x\exists y\,\psi(x,y,c,e,g).
  \]
  By Proposition~\ref{prop:regularRealizationTrivialD}, however, there is a regular \(\BS(m,n)\)-field extension \(L/K\), still with trivial \(D\)-action, in which \(\psi(x,y,c,e,g)\) has a solution.
  Since \(D\) acts trivially on \(L\), its action factors through \(Q\), making \(L/K\) a regular \(Q\)-field extension.
  Since \(K\) is p.e.c. as a \(Q\)-field, the formula already has a solution in \(K\), a contradiction.
\end{proof}

\begin{corollary}\label{cor:BSObstruction}
  Let \(K\) be an \(\aleph_0\)-saturated p.e.c. \(\BS(m,n)\)-field and let \(q\) be admissible for \(\BS(m,n)\) over \(K\).
  Then \(K\) does not satisfy \(\theta_q\).
\end{corollary}

\begin{proof}
  Apply Theorem~\ref{th:BSQuotient} to the identity map of \(\BS(m,n)\).
\end{proof}

\section{Groups Containing a Baumslag--Solitar Group}\label{sec:overgroups}

Let
\[
  \BS(m,n)=\langle\sigma,\tau\rangle
\]
be an embedded Baumslag--Solitar subgroup of a group \(G\).

\begin{lemma}\label{lem:iteratedRelation}
  For every integer \(\nu\geq1\), put
  \[
    H_\nu:=\langle\sigma,\tau^\nu\rangle\leq\BS(m,n).
  \]
  Then
  \[
    \tau^\nu\sigma^{m^\nu}=\sigma^{n^\nu}\tau^\nu,
  \]
  and hence there is a surjective homomorphism
  \[
    \varphi_\nu:\BS(m^\nu,n^\nu)\twoheadrightarrow H_\nu.
  \]
  If \(\sigma_\nu,\tau_\nu\) denote the standard generators of the source, then
  \[
    \varphi_\nu(\sigma_\nu)=\sigma,\quad \varphi_\nu(\tau_\nu)=\tau^\nu.
  \]
  Let \(\beta_\nu\) be the intrinsic 1-cocycle and \(\ell_\nu\) the \(\tau_\nu\)-exponent-sum homomorphism of \(\BS(m^\nu,n^\nu)\).
  Identifying
  \[
    \Z\left[\frac1{m^\nu n^\nu}\right]=\Z\left[\frac1{mn}\right],
  \]
  one has
  \[
    \beta\circ\varphi_\nu=\beta_\nu,\quad \ell\circ\varphi_\nu=\nu\ell_\nu.
  \]
  Consequently,
  \[
    \ker\varphi_\nu\subseteq\ker\beta_\nu\cap\ker\ell_\nu.
  \]
\end{lemma}

\begin{proof}
  For every integer \(k\), the defining relation implies
  \[
    \tau\sigma^{mk}\tau^{-1}=\sigma^{nk}.
  \]
  Iterating this identity \(\nu\) times gives
  \[
    \tau^\nu\sigma^{m^\nu}\tau^{-\nu}=\sigma^{n^\nu},
  \]
  which yields the stated surjection.
  The two displayed identities are checked on \(\sigma_\nu,\tau_\nu\): \(\beta\circ\varphi_\nu\) and \(\beta_\nu\) both take the values \(1,0\), whereas \(\ell\circ\varphi_\nu\) and \(\nu\ell_\nu\) both take the values \(0,\nu\).
  If \(\gamma\in\ker\varphi_\nu\), the first identity gives \(\beta_\nu(\gamma)=0\), while the second gives \(\nu\ell_\nu(\gamma)=0\).
  Since \(\Z\) is torsion-free, one also has \(\ell_\nu(\gamma)=0\).
\end{proof}

\begin{theorem}\label{th:overgroupBSmn}
  If \(\BS(m,n)\leq G\) for some non-zero integers \(m,n\), then \(G\TCF\) does not exist.
\end{theorem}

\begin{proof}
  Suppose toward a contradiction that \(G\TCF\) exists.
  Choose a prime \(q\nmid mn\), and put
  \[
    \nu:=q-1.
  \]
  By Lemma~\ref{lem:iteratedRelation}, there is a surjective homomorphism
  \[
    \varphi_\nu:\BS(m^\nu,n^\nu)\twoheadrightarrow H_\nu:=\langle\sigma,\tau^\nu\rangle
  \]
  whose kernel is contained in \(\ker\beta_\nu\cap\ker\ell_\nu\).
  The cyclotomic exponent of the source is
  \[
    w_\nu\equiv n^\nu(m^\nu)^{-1}=\left(\frac nm\right)^\nu\equiv1\pmod q
  \]
  by Fermat's little theorem.

  Let \(\zeta\) be a primitive \(q\)-th root of unity and give
  \[
    K_0:=\Q(\zeta)
  \]
  the trivial \(G\)-action.
  Under the assumed existence of \(G\TCF\), the \(G\)-field \(K_0\) embeds into an existentially closed \(G\)-field; take a sufficiently saturated elementary extension \(K\).
  Then \(K\) is \(\aleph_0\)-saturated and p.e.c., and every element of \(G\) fixes \(\zeta\).

  Restrict the action to \(H_\nu\).
  By Corollary~\ref{cor:Descent}, this reduct is p.e.c.
  It remains \(\aleph_0\)-saturated: a type over a finite parameter set in the \(H_\nu\)-field language is a partial type in the \(G\)-field language, and every finite subset is realized in \(K\).
  Relative to \(\varphi_\nu\), it satisfies \(\theta_q\): \(\sigma_\nu\) acts as \(\sigma\) and \(\tau_\nu\) as \(\tau^\nu\), both fixing \(\zeta\), while the required cyclotomic exponent is \(w_\nu=1\).
  Theorem~\ref{th:BSQuotient} gives a contradiction.
\end{proof}

\begin{corollary}\label{cor:noBSTCF}
  For every non-zero \(m,n\), \(\BS(m,n)\TCF\) does not exist.
\end{corollary}

\begin{proof}
  Apply Theorem~\ref{th:overgroupBSmn} with \(G=\BS(m,n)\).
\end{proof}

\begin{corollary}\label{cor:Z2Overgroup}
  If \(\Z^2\leq G\), then \(G\TCF\) does not exist.
\end{corollary}

\begin{proof}
  Use \(\Z^2\cong\BS(1,1)\) in Theorem~\ref{th:overgroupBSmn}.
\end{proof}

\begin{example}[Further examples]\label{ex:furtherExamples}
  \begin{enumerate}[label=\textup{(\roman*)}]
    \item Since \(\Z^2\leq\Q^2\), \(\Q^2\TCF\) does not exist.
    \item Thompson's group \(F\) has the infinite presentation
          \[
            F=\langle x_0,x_1,x_2,\ldots\mid x_i^{-1}x_jx_i=x_{j+1}\ (i<j)\rangle.
          \]
          The relations for \((i,j)=(0,2)\) and \((1,2)\) show that \(x_0x_1^{-1}\) commutes with \(x_2\), and their images in \(F_{\mathrm{ab}}\cong\Z^2\) are linearly independent.
          Hence
          \[
            \langle x_0x_1^{-1},x_2\rangle\cong\Z^2;
          \]
          see \cite[Theorems~3.1 and 4.1]{CannonFloydParry1996}.
          Therefore \(F\TCF\), \(T\TCF\), and \(V\TCF\) do not exist.
    \item Higman's group
          \[
            H=\langle a,b,c,d\mid a^{-1}ba=b^2,\ b^{-1}cb=c^2,\ c^{-1}dc=d^2,\ d^{-1}ad=a^2\rangle
          \]
          contains the subgroup \(\langle a,b\rangle\cong\BS(1,2)\), with the standard generators corresponding to \(\sigma=b\) and \(\tau=a^{-1}\); cyclically, the other adjacent pairs give three further copies.
          Hence \(H\TCF\) does not exist; see \cite[\S2.1]{RivasTriestino2019}.
  \end{enumerate}
\end{example}

\begin{example}[The group \(\BS(2,3)\)]
  This example explains why Theorem~\ref{th:BSQuotient} is needed in the proof of Theorem~\ref{th:overgroupBSmn}.
  The iterated relation gives only a surjection from a Baumslag--Solitar group onto \(H_\nu\), and this map need not be injective; thus \(H_\nu\) cannot in general be treated as a Baumslag--Solitar group itself.

  For \(\BS(2,3)=\langle\sigma,\tau\rangle\), take \(q=5\) and \(\nu=4\).
  Then
  \[
    \left(\frac32\right)^4=\frac{81}{16}\equiv1\pmod5,
  \]
  and Lemma~\ref{lem:iteratedRelation} gives
  \[
    \varphi_4:\BS(16,81)\twoheadrightarrow H_4:=\langle\sigma,\tau^4\rangle\leq\BS(2,3).
  \]
  We show that this surjection is not injective.
  Let \(\sigma_4,\tau_4\) denote the standard generators of \(\BS(16,81)\), and set
  \[
    \delta:=\tau_4\sigma_4^{-8}\tau_4^{-1}\sigma_4^{-3}\tau_4\sigma_4^8\tau_4^{-1}\sigma_4^3.
  \]
  Viewed as an HNN extension with base group \(\langle\sigma_4\rangle\), the word \(\delta\) is reduced: the exponents \(\pm8\) are not divisible by \(16\), and \(-3\) is not divisible by \(81\).
  Hence \(\delta\neq1\) by Britton's lemma \cite[Chapter~IV]{LyndonSchupp1977}.

  On the other hand, the displayed surjection sends \(\sigma_4\) to \(\sigma\) and \(\tau_4\) to \(\tau^4\).
  In \(\BS(2,3)\), iterating the defining relation three times gives
  \[
    \tau^3\sigma^8\tau^{-3}=\sigma^{27},
  \]
  so
  \[
    \tau^4\sigma^8\tau^{-4}=\tau\sigma^{27}\tau^{-1}.
  \]
  Also \(\sigma^3=\tau\sigma^2\tau^{-1}\).
  Thus \(\tau^4\sigma^8\tau^{-4}\) and \(\sigma^3\) are conjugates by \(\tau\) of powers of \(\sigma\), and hence commute.
  Therefore the image of \(\delta\) is
  \[
    (\tau^4\sigma^8\tau^{-4})^{-1}\sigma^{-3}(\tau^4\sigma^8\tau^{-4})\sigma^3=1.
  \]
  Thus the kernel is non-trivial.
  Nevertheless, Lemma~\ref{lem:iteratedRelation} gives
  \[
    \ker\varphi_4\subseteq\ker\beta_4\cap\ker\ell_4,
  \]
  so Theorem~\ref{th:BSQuotient} applies despite the non-trivial kernel.
\end{example}

\end{document}